\documentclass[preprint]{imsart}

\RequirePackage[OT1]{fontenc}
\RequirePackage{amsthm,amsmath,times}
\RequirePackage[numbers]{natbib}
\RequirePackage[colorlinks,citecolor=blue,urlcolor=blue]{hyperref}

\startlocaldefs
\numberwithin{equation}{section}
\theoremstyle{plain}
\newtheorem{thm}{Theorem}[section]
\newtheorem{lem}{Lemma}[section]
\endlocaldefs

\newcommand{\tilZ}{\tilde{Z}}
\newcommand{\Tf}{T_{f}}
\newcommand{\Tb}{T_{b}}
\newcommand{\Tbf}{\tilde{T}}
\newcommand{\tZ}{\theta'Z}
\newcommand{\tz}{\theta'z}
\newcommand{\ttZ}{\theta'\tilde{Z}}

\newcommand{\Ut}{U(\theta)}
\newcommand{\Uct}{U^{c}(\theta)}
\newcommand{\Uctzero}{U^{c}(\theta_{0})}
\newcommand{\Utzero}{U(\theta_{0})}

\newcommand{\mug}{\int S_{g}(v)dv}
\newcommand{\ettzeroz}{e^{(\theta-\theta_{0})'z}}

\newcommand{\ettzeroZi}{e^{(\theta-\theta_{0})'\tilde{Z}_{i}}}

\newcommand{\Ui}{U_{i}}
\newcommand{\Uci}{U^{c}_{i}}

\newcommand{\tzero}{\theta_{0}}

\newcommand{\tilC}{\tilde{C}}

\begin{document}

\begin{frontmatter}
\title{Efficient Estimation of Accelerated Lifetime Models under Length-Biased Sampling}
\runtitle{AFT Model Estimation in Length-Biased Sampling}

\begin{aug}
\author{\fnms{Pourab} \snm{Roy}\thanksref{t1}\ead[label=e1]{pourab.roy@fda.hhs.gov}}

\address{US Food and Drug Administration\\
	10903 New Hampshire Avenue Silver Spring, Maryland 20993\\
\printead{e1}}

\author{\fnms{Jason P.} \snm{Fine}\ead[label=e2]{jfine@email.unc.edu}}
\and
\author{\fnms{Michael R.} \snm{Kosorok}\ead[label=e3]{kosorok@bios.unc.edu}}

\address{Department of Biostatistics, University of North Carolina at Chapel Hill,\\ 3101 McGavran-Greenberg Hall, CB 7420
	Chapel Hill, North Carolina 27599-7420\\
	\printead{e2,e3}}

\thankstext{t1}{This work was done prior to the author joining the FDA and does not represent the official position of the FDA}

\runauthor{P. Roy et al.}

\affiliation{University of North Carolina at Chapel Hill}

\end{aug}

\begin{abstract}
In prevalent cohort studies where subjects are recruited at a cross-section, the time to an event may be subject to length-biased sampling, with the observed data being either the forward recurrence time, or the backward recurrence time, or their sum. In the regression setting, it has been shown that the accelerated failure time model for the underlying event time is invariant under these observed data set-ups and can be fitted using standard methodology for accelerated failure time model estimation, ignoring the length-bias. However, the efficiency of these estimators is unclear, owing to the fact that the observed covariate distribution, which is also length-biased, may contain information about the regression parameter in the accelerated life model. We demonstrate that if the true covariate distribution is completely unspecified, then the naive estimator based on the conditional likelihood given the covariates is fully efficient.
\end{abstract}

\begin{keyword}
	\kwd{accelerated failure time model}
	\kwd{backward recurrence time}
	\kwd{forward recurrence time}
	\kwd{length-biased time}
\end{keyword}
\tableofcontents
\end{frontmatter}

\section{Introduction}

Many study designs involve cross-sectional sampling, which may lead to length-biased sampling of a time to event $T$. In a prospective cohort study where the initiation time for the event is unknown and subjects are followed prospectively, a right-censored forward recurrence time $\Tf$ is observed. This occurs, for example, in HIV seroprevalence studies \citep{Brookmeyer87}, where time to AIDS following infection with HIV is of interest but the infection time is unknown. If the initiation time is known, but there is no follow-up, the backward recurrence time $\Tb$ is observed. This current duration design \citep{Keiding02} has been employed in pregnancy surveys, where current trier couples provide the length of an ongoing attempt at pregnancy, and in mover-stayer models \citep{Yamaguchi03}. If there is both an initiation time and a follow-up, a biased event time $T_{\mathrm{LB}} = \Tf + \Tb$ may be observed. When the sampling time is known, such data are commonly analyzed using methods for left-truncated data, where one conditions on the lack of an event prior to the sampling time, that is, $T_{\mathrm{LB}}>\Tb$.

In the aforementioned data set-ups, only subjects who have experienced an initiating event prior to sampling can potentially be sampled, and the sample is biased towards larger values of $T$. If one assumes that the rate of the initiating event is stationary over time, e.g., is a homogeneous Poisson process, then the sampling time falls uniformly in the interval between the initiation and the event times \citep{Cox69,Vardi82}. Letting $F_{T}$ denote the distribution of $T$, the length-biased version $T_{\mathrm{LB}}$ has distribution $F_{\mathrm{LB}}(t)=\int_{0}^{t} u dF_{T}(u)\mu_{T}^{-1}$, $\quad t \ge 0$,
where $\mu_{T} = \int_{0}^{\infty} u dF_{T}(u)$.
Under the uniform sampling time assumption \citep{Cox69,VanEs00,Keiding02}, $\Tf = T_{\mathrm{LB}}V$, where $V$ is uniform(0,1) and independent of $T_{\mathrm{LB}}$.  Thus, both $\Tf$ and $\Tb$ will have the same density function. Since the density is the same, we use $\Tbf$ to denote both $\Tf$ and $\Tb$. The density of $\Tbf$ is
\begin{equation}
g_{\Tbf}(t) = S_{T}(t)\mu_{T}^{-1}, t \in (0,\infty),
\label{eq1.1}
\end{equation}
where $S_{T} = 1-F_{T}$ is the survival function of $T$. This well-known result, given in expression (2) of \citep{Huang11}, can be derived from the uniformity of $V$, which yields that the conditional density of $\Tf$ given $T_{\mathrm{LB}}$ has the form
\begin{equation}
f_{\Tf | T_{\mathrm{LB}}}(s|t)=I(0<s<t)t^{-1},
\label{eq1.1new}
\end{equation}
yielding the joint density
$f_{\Tf,T_{\mathrm{LB}}}(s,t)=I(0<s<t)f_T(t)\mu_T^{-1}$, from which (\ref{eq1.1}) follows.

In the presence of a $p \times 1$ covariate vector $Z$ with density $h$, one may formulate the effect of $Z$ on the underlying event time $T$ via the accelerated lifetime model \begin{eqnarray}
T = e^{\tZ}U,
\label{eq1.2}
\end{eqnarray}
where $\theta$ is a $p \times 1$ regression parameter and $U$ is a non-negative random variable with density $g$, survival function $S$ and hazard function $\lambda(u)=g(u)/S(u)$. In the semiparametric version of (\ref{eq1.2}), with the distribution of $U$ completely unspecified, efficient estimation of $\theta$ without length-bias is achievable with and without right censoring \citep{Zeng07}.

The observed covariate is also subject to length-biased sampling. Using arguments in \cite{Chen09} and \cite{Mandel10}, we
first obtain that a consequence of (\ref{eq1.2}) is that the joint density of $T_{LB}$ and observed covariate $\tilde{Z}$
is proportional to $te^{-\theta'z}g(e^{-\theta'z}t)h(z)$ due to the length bias, yielding the joint density
\begin{equation}
f_{T_{LB},\tilde{Z}}(t,z)=te^{-\theta'z}g(e^{-\theta'z}t)\mu_g^{-1}h(z)/\int e^{\theta'z}h(z)dz, \label{eq1.2new}
\end{equation}
where the integral is over the range of $Z$ and $\mu_g$ is the mean associated with density $g$. This means that 
$T_{LB}$ satisfies an accelerated life model, given by 
\begin{equation}
T_{\mathrm{LB}} = e^{\theta ' Z_{\mathrm{LB}}}U_{\mathrm{LB}},
\label{eq1_lb}
\end{equation}
where $Z_{\mathrm{LB}}$ is independent of $U_{\mathrm{LB}}$ with density of the form 
\begin{equation*}
h_{Z_{\mathrm{LB}}}(z) = e^{\tz}h(z)/\int e^{\tz}h(z)dz
\end{equation*}
 and $U_{\mathrm{LB}}$ has the possibly non-monotone density 
$g_{U_{\mathrm{LB}}}(u) = ug(u)\mu_g^{-1}$, $u\in(0,\infty)$.

The relation $\Tbf=T_{LB}V$ still holds for a uniform$(0,1)$ $V$ independent of $T_{LB}$, so we can multiply (\ref{eq1.2new}) by (\ref{eq1.1new}), integrate over $t$, and replace $s$ with $t$, to obtain the joint density of $\tilde{T}$ and observed covariate $\tilde{Z}$:
\begin{equation}
f_{\Tbf,\tilde{Z}}(t,z) = \frac{e^{-\tz}S(e^{-\tz}t)}{\mu_{g}} \times \frac{e^{\tz}h(z)}{\int e^{\theta 'u}h(u)du}.
\label{eq1.3}
\end{equation}
Thus, if $T$ follows model (\ref{eq1.2}), then the distribution of $\Tbf$ also follows an accelerated lifetime model 
\begin{equation}
\tilde{T} = e^{\ttZ}\tilde{U},
\label{eq1.4}
\end{equation}
where $\tilde{Z}$ has a density of the form $h_{\tilde{Z},\theta}(z) = e^{\tz}h(z)/\int e^{\theta 'u}h(u)du$ and $\tilde{U}$ has monotone density $g_{\tilde{U}}(u) = S(u)/\int_{0}^{\infty}S(v)dv$.
Thus the accelerated lifetime structure is maintained in models for the forward and backward recurrence times, as discussed in \cite{Keiding11}, as well as for the length-biased time.

Since the conditional distributions of $\tilde{T}$ and $T_{\mathrm{LB}}$ satisfy accelerated lifetime models, existing estimation procedures may naively be applied to obtain semiparametric estimators for $\theta$. However, as the marginal distribution of the observed covariates depends on the parameter $\theta$, it is unclear whether estimators derived from the conditional distributions of $\tilde{T}$ and $T_{\mathrm{LB}}$ will be fully efficient. Estimation using observed covariates has been considered in \cite{Bergeron11}, and \cite{Chan13}. Under restrictive assumptions on $h$, for example, moment restrictions, improved estimation is possible. However, a comprehensive study of such issues with completely unspecified covariate distribution and right censoring has not been undertaken for general length-biased sampling. 

The main contribution of this paper is to show that a naive efficient estimator which ignores the dependence on $\theta$ in the marginal covariate distribution is still efficient for estimation of the regression parameter in length-biased and recurrence time data. Hence, the standard techniques that are used for estimation in the accelerated failure time models can also be applied in these cases without loss of information. We provide the theoretical derivation of the efficient score in Section \ref{sec:effsc}, simulation results and data analysis in Section \ref{sec:sim} and \ref{sec:dat} and conclude with some discussions in Section \ref{sec:disc}. 

\section{Efficient Scores and Estimation}\label{sec:effsc}

We start by defining some of the important assumptions and notation used in the paper, along with a description of some of the tangent spaces used in deriving the efficient score. The concepts and notation closely follow that given in Chapters 3 and 18 of \cite{K08}.

Let $\mathcal{G}$ be the class of all density functions on $\Re^{+}$ and $\mathcal{H}$ be the class of all density functions on $\Re^{p}$. The semiparametric model for the core accelerated lifetime model 
is given by
\begin{equation*}
\mathcal{P}^{*}=\left\{ P_{\theta,g,h}^{*} : \ \theta\in\Theta, g\in\mathcal{G}, h\in \mathcal{H}\right\},
\end{equation*}
where the distribution $P_{\theta,g,h}^{*}$ has a density with respect to an absolutely continuous measure $\nu$, 
\begin{equation*}
\frac{d P_{\theta,g,h}^{*}}{d\nu} (t,z)= e^{(\theta-\theta_{0})'z} g\{e^{(\theta-\theta_{0})'z}t\}h(z),
\end{equation*}
where $\theta_{0}$ is the true parameter value. For the accelerated failure time model of the recurrence times, the semiparametric model is 
\begin{equation}
\mathcal{P}=\left\{ P_{\theta,S_g, h} : \ \theta\in\Theta, g\in\mathcal{G}, h\in \mathcal{H}'\right\},
\label{eq2.2.1}
\end{equation}
where $S_g(u) = \int_{u}^{\infty}g(v)dv$ for $g \in \mathcal{G}$, and 
\begin{equation*}
\mathcal{H}' = \left\{ h: h\in \mathcal{H}, \int e^{\tz}h(z)d(z) < \infty, 
\int z^{\otimes 2}e^{\tz}h(z)dz < \infty, \theta\in\Theta\right\}.
\end{equation*}
We assume that $\Theta$ is a compact subset of $\mathcal{R}^{p}$. Further, we assume $P_{\theta,S_{g},h}$ has density  
\begin{equation*}
\frac{dP_{\theta,S_{g},h}}{d\nu}(t,z)=\frac{\ettzeroz S_{g}\{\ettzeroz t\} }{\mug} \times \frac{e^{\tz}h(z)}{\int e^{\tz}h(z)dz}.
\end{equation*}
Define $\mathcal{S} = \left\{ S_{g}: g \in \mathcal{G}\right\}$. Let the true distribution be $P_{0}=P_{\theta_{0},S_{0},h_{0}}$ with $S_{0}=S_{g_{0}}$. Define the separate submodels for each parameter, holding
the other parameters fixed, as
$\mathcal{P}_{\theta}  =  \{P_{\theta,S_{0},h_{0}}: \theta\in\Theta\}$, 
$\mathcal{P}_{S}  =  \{P_{\theta_{0},S,h_{0}}: S\in\mathcal{S}\}$ and
$\mathcal{P}_{h}  =  \{P_{\theta_{0},S_{0},h}: h\in\mathcal{H}'\}$.

Let $\dot{\mathcal{P}}_{\theta}$, $\dot{\mathcal{P}}_{S}$ and $\dot{\mathcal{P}}_{h}$ be the tangent spaces for $\mathcal{P}_{\theta}, \mathcal{P}_{S} \textrm{ and } \mathcal{P}_{h}$ at $P_{0}=P_{\theta_{0},S_{0},h_{0}}$. By definition of tangent spaces in Chapter 18 of \cite{K08}, these are all closed subsets of $L_2^0(P_0)$, where $L_2^0(P_0)$ denotes square-integrable functions integrating to zero with respect to $P_0$. For a density $\nu$, let $L_2(\nu)$ denote the space of square-integrable functions with respect to the measure $\int\nu$. For a survival function $S$, we similarly let $L_2(S)$ denote the space of square-integrable functions with respect to the measure $\int S$, even though $S$ may not integrate to 1. Tangent spaces for a given model represent the set of all likelihood score functions for one-dimensional submodels of the given model. The three tangent spaces just defined represent the likelihood based scores used to estimate the parameter given in the subscript while holding the remaining parameters fixed at their true values. Let $\dot{l}_{\theta}$ be the ordinary score for $\theta$ when $S$ and $h$ are fixed. Then the efficient score function $\tilde{l}_{\theta}\in \{ L_{2}^{0}(P_{0})\}^{p}$ for $\theta$ in the full model $\mathcal{P}$ at $P_{0}$ is $\tilde{l}_{\theta}=\dot{l}_{\theta} - \Pi_{0}(\dot{l}_{\theta}\mid \dot{\mathcal{P}}_{S} + \dot{\mathcal{P}}_{h})$, where $\Pi_{0}(l\mid \mathcal{W})$ denotes the orthogonal projection of $l$ onto the linear span of $\mathcal{W}$ \citep{BKRW}. 

\subsection{Inference for Forward and Backward Recurrence Times}

We now calculate the efficient score and information using the recurrence time $\tilde{T}$ potentially subject to right censoring, which covers both the forward and backward recurrence time settings. For the $i$-th individual, we observe ($\tilde{T}_i\wedge \tilde{C}_i,\delta_i,\tilde{Z}_i$), where $\tilde{T}_i$ is the recurrence time of the $i$-th individual, $\tilde{C}_i$ is the time of right censoring, which is assumed to be independent of the recurrence time conditionally given the covariates, $\delta_i$ is the indicator of whether the event time is observed and $\tilde{Z}_i$ is the $p \times 1$ observed covariate. Theorem \ref{thm:1} below demonstrates that in this setting, the efficient score equals that of the naive efficient estimator based on the conditional likelihood given the covariates.

For right-censored $\Tbf$, we assume $\Theta$ is a compact set in $\Re^{p}$, and that $\theta_{0}$ belongs in the interior of $\Theta$. For fixed but arbitrary $\theta$, we define our semiparametric model in terms of the distribution of $\Ut = e^{-\ttZ}\Tbf =e^{-(\theta-\theta_{0})'\tilZ} \tilde{U}$ and the corresponding censored variable $\Uct= e^{-\ttZ} \tilC$. The conditional density of $\Ut$ given $\tilZ = z$ is thus
\begin{equation*}
g_{\Ut}(u) = \frac{\ettzeroz S\{ e^{(\theta-\theta_{0})'z}u\} }{\int S(v)dv},
\end{equation*}
while the conditional hazard is
\begin{equation*}
\lambda_{\Ut}(u) = \frac{S\{ e^{(\theta-\theta_{0})'z}u \}}{\int_{u}^{\infty} S\{\ettzeroz w\}dw}, \quad u \in (0,\infty).
\end{equation*}
Given $\tilZ$ the density of $\Ut$ is monotone decreasing. We now state our assumptions:
\begin{itemize}
	\item[A1:] $\Tbf$ and $\tilC$ are independent given $\tilZ$;
	\item[A2:] The distribution of $\tilC$ is independent of the parameters $(\theta,S,h)$, and the distribution of $Z$ is independent of the parameters $(\theta,S)$;
	\item[A3:] $\int S(v)dv < \infty$;
	\item[A4:] $E_{g_{U(\theta)}}\left\{U^{2}\lambda(U)\right\} = \int u^{2}g_{U(\theta)}^{2}S(u)^{-1}du < \infty$.
\end{itemize}

The last assumption is needed to ensure that the density of $U(\theta )$ has finite Fisher information about $\theta$. 
The next theorem gives that the efficient score equals that from the naive efficient estimator.
\begin{thm}
	\label{thm:1}
	Suppose that the covariate vector $\tilZ$ is almost surely bounded. Define
	\begin{equation}
	M(t) = I\{\Ut\le t\} - \int_{0}^{t} I\{\Ut > s\}\lambda_{\Ut}(s)ds
	\label{eq2.2}
	\end{equation}
	and
	\begin{equation*}
	Ra(t) = a(t) - \frac{\int_{t}^{\infty}a(u)S(u)du}{\int_{t}^{\infty}S(u)du},\quad \textrm{ for } a\in L_{2}^{0}(S).
	\end{equation*}
	Then under (A1)--(A4) and with $\phi(u) = 1- u g(u)/S(u)$, the ordinary score for $\theta$ at $\theta=\tzero$ is
	\begin{equation}
	\dot{l}_{\tzero}= \tilZ\int_{0}^{\Uctzero} R\phi(s)dM(s) + (\tilZ - E\tilZ),
	\label{eq2.3}
	\end{equation} 
	the tangent space $\dot{\mathcal{P}}_{S}$ for $S$ is $\{\dot{l}_{S}b: b\in L_{2}^{0}(S)\}$ where the score operator $\dot{l}_{S}$ for $S$ is given by
	\begin{equation}
	\dot{l}_{S}b = \int_{0}^{\Uctzero} Rb(s)dM(s),
	\label{eq2.4}
	\end{equation}
	the tangent space for $h$ is $\{k: k\in L_{2}(h), \int k(z)e^{\theta_{0}'z}h(z)dz=0\}$, and the efficient score for $\theta$ at $\theta=\tzero$ is
	\begin{equation}
	\tilde{l}_{\theta,S} = \int_{0}^{\Uctzero}[\tilZ - E\{\tilZ | \Uctzero \ge s\}]R \phi(s)dM(s).
	\label{eq2.5}
	\end{equation}
\end{thm}

\begin{proof}
	The likelihood for one observation $(\Ui \wedge \Uci,\delta_{i},\tilZ_{i})$ is given by
	\begin{eqnarray*}
		l(\theta)=\left\{g_{\Ut}(U_{i})\right\}^{\delta_{i}}\left\{\int_{\Uci}^{\infty}g_{\Ut}(u)du\right\}^{1-\delta_{i}}h_{Z,\theta}(\tilZ_{i}).
	\end{eqnarray*}
	Taking log and differentiating with respect to $\theta$, we obtain the ordinary score for $\theta$ at $\theta=\theta_0$, 
	\[\dot{l}_{\theta_0}=\tilZ_i\left[\delta_{i} \phi(\Ui) + (1-\delta_{i}) E\left\{\phi(Ut)\mid\Ut > \Uci\right\}\right]+(\tilZ_{i} - E \tilZ).\]
	The expression in \ref{eq2.3} for the ordinary score function for $\theta$ can be derived by noting that the quantity in brackets on the right hand side of the above expression is a stochastic integral with respect to the counting-process martingale in \ref{eq2.2} \citep{BKRW}, using proposition A.3.6 in \cite{BKRW}. 
	
	Next, we can conclude from the Lemma~\ref{lem1} below that the tangent space $\dot{\mathcal{Q}}_{S}$ for $S$ can be considered the maximal tangent space $L_{2}^{0}(S)$. Hence the tangent space
	for $S$ can be expressed through the one dimensional submodels $\eta\mapsto S_{\eta}(t)=(1+\eta b(t))S(t)$ for any
	$b\in L_2^0(S)$, which yield the one dimensional baseline hazard submodels 
	\begin{equation*}
	\eta\mapsto \lambda_{\eta}(t)=\frac{S_{\eta}(t)}{\int_t^{\infty}S_{\eta}(v)dv}.
	\end{equation*}
	Differentiating the likelihood with respect to $\eta$ and setting $\theta=\theta_0$ now yields the score given in~\ref{eq2.4}. In order to find $\Pi_{0}(\dot{l}_{\tzero}\rvert \dot{\mathcal{P}}_{S}) = \dot{l}_{S}b^{*}$ we find $b^{*} \in L^{0}_{2}(S)$ such that $\dot{l}_{\tzero} - \dot{l}_{S}b^{*} \ \perp \ \dot{l}_{S}b$ for all $b\in L_{2}^{0}(S)$. That is $E\left\{ \left(\dot{l}_{\tzero} - \dot{l}_{S}b^{*}\right)\dot{l}_{S}b\right\}=0$. Note that $\dot{l}_{\tzero} - \dot{l}_{S}b^{*} = \int _{-\infty}^{\Uctzero} (\tilZ R\phi - Rb^{*})dM(s) + (\tilZ - E\tilZ)$. Conditioning on $\tilZ$ and $\Uctzero$ and using the fact that $\Utzero$ is  distributed independently of $\tilZ$ and $\Uctzero$ we obtain
	\begin{align*}
	& E\left\{ \left(\dot{l}_{\tzero} - \dot{l}_{S}b^{*}\right)\dot{l}_{S}b\right\} \\
	& =  E E\left\{(\dot{l}_{\tzero} - \dot{l}_{S}b^{*})\dot{l}_{S} b \mid \tilZ,\Uctzero \right\}\\
	& =  E E\left\{ \int_{0}^{\Uctzero}(\tilZ R\phi(s)-R b^{*}(s))Rb(s)I\{\Utzero \ge s\}\lambda_{\Utzero}(s)ds \mid \tilZ,\Uctzero\right\} \\
	& =  E\left\{ \int_{0}^{\Uctzero}(\tilZ R\phi(s)-R b^{*}(s))Rb(s)dF_{\Utzero}(s)\right\} \\
	& = \int \left\{E(\tilZ I\{\Uctzero \ge s\})R\phi(s) - EI\{\Uctzero \ge s\}Rb^{*}(s)\right\}Rb(s)dF_{\Utzero}(s).
	\end{align*}
	The second equality above is obtained by using the result that if $Y_{i} = \int f_{i}dM, i =1,2$, then 
	\begin{equation}
	EY_{1}Y_{2} = E\int f_{1} f_{2} d\langle M,M \rangle =E\int f_{1}(s)f_{2}(s)I\{\Utzero \ge s\}d\Lambda(s).\nonumber
	\end{equation}
	Thus $E\left\{ \left(\dot{l}_{\tzero} - \dot{l}_{S}b^{*}\right)\dot{l}_{S}b\right\} =0$ for all $b \in L_{2}^{0}(S)$ if
	\begin{equation*}
	Rb^{*}(s) = \frac{E\{\tilZ I(\Uctzero \ge s)\}}{EI(\Uctzero \ge s)}R\phi(s) = E(\tilZ | \ \Uctzero \ge s)R\phi(s).
	\end{equation*}
	Thus the projection of $\dot{l}_{\theta}$ on $\dot{\mathcal{P}}_{s}$ is given by
	\begin{equation}
	\Pi_{0}(\dot{l}_{\theta_{0}}\rvert \dot{\mathcal{P}}_{S})=\int_{0}^{\Uctzero} E(\tilZ | \ \Uctzero \ge s)R\phi(s)dM(s).
	\label{eq2.6}
	\end{equation}
	
	Now for finding $\dot{\mathcal{P}}_{h}$ for $h \in \mathcal{H}'$, we consider the one-parameter path $\eta\mapsto h_{\eta} = (1+\eta k)h$, where $k\in L_{2}^{0}(h)$. The score operator for $h$ is given by
	\begin{displaymath}
	\dot{l}_{h}k = k - \frac{\int e^{\theta_{0}'z}k(z)h(z)dz}{\int e^{\theta_{0}'z}h(z)dz} \equiv m(z),
	\end{displaymath}
	for $k\in L_{2}^{0}(h)$. Note that $\int k(z)e^{\theta_{0}'z}h(z)dz = 0$. If $h$ is unrestricted then the tangent space can be taken to be the orthocomplement of the linear span of $e^{\theta_{0}'z}$, i.e., $[e^{\theta_{0}'z}]^{\perp}$ in $L_{2}(h)$. Since $\Utzero$ is distributed independently of $\tilZ$ and $\Uctzero$, $E_{0}\{k(\tilZ)\dot{l}_{S}b\} = 0$ for any $k\in [e^{\theta_{0}'z}]^{\perp}$ and $\dot{l}_Sb \in \dot{\mathcal{P}}_{S}$, i.e., $\dot{\mathcal{P}}_{S} \perp \dot{\mathcal{P}}_{h}$. Since $(z - E\tilZ) \in [e^{\theta_{0}'z}]^{\perp}$, we obtain
	\begin{equation}
	\Pi_{0}(\dot{l}_{\theta_{0}} | \ [e^{\theta_{0}'z}]^{\perp})  = z - E\tilZ .
	\label{eq2.7}
	\end{equation}
	Now replacing $z$ with $\tilZ$ and subtracting \ref{eq2.6} and \ref{eq2.7} from \ref{eq2.3} yields the efficient score given in \ref{eq2.5}. The accelerated failure time model
	for $\tilde{T}$ given $\tilde{Z}$ is equivalent to the log-linear model $Y=\log(\tilde{T})=-\theta'\tilde{Z}+\epsilon$, where $\epsilon$ has hazard function 
	\begin{eqnarray}
	\lambda(t)&=&\lambda_{\tilde{T}}(e^t)e^t,\label{eq3.1}
	\end{eqnarray}
	and $\lambda_{\tilde{T}}(u)$ is the baseline hazard for $\tilde{T}$. In the current setting,
	\begin{eqnarray}
	\lambda_{\tilde{T}}(u)&=&\frac{S(u)}{\int_u^{\infty}S(v)dv}.\label{eq3.2}
	\end{eqnarray} 
	This model is the same as the linear regression model for $Y$ but with a sign change on $\theta$. 
	The efficient score for the linear regression model under right-censoring is given in Expression (27) on Page 149 of \cite{BKRW} and has the same form as \ref{eq2.5}, except for changes in 
	parameter and variable notation. Specifically, the function $R\phi(u)$ in \ref{eq2.5} equals the negative of $R\psi(t)=-\dot{\lambda}(t)/\lambda(t)$ defined in Expression (23) of \cite{BKRW}, after replacing $t$ with $\log(u)$, where the negative is due to the sign change. To see this, note that 
	\begin{eqnarray*}
		-R\psi(t)&=&\frac{\dot{\lambda}(t)}{\lambda(t)}\;=\; 1+\frac{\dot{\lambda}_{\tilde{T}}(e^t)e^t}{\lambda_{\Tilde{T}}(e^t)}\\
		&=&1+\frac{\dot{\lambda}_{\tilde{T}}(u)u}{\lambda_{\tilde{T}}(u)}\;=\;1-\frac{g(u)}{S(u)}+\frac{S(u)}{\int_u^{\infty}S(v)dv}\\
		&=&R\phi(u).
	\end{eqnarray*}
	The first row follows from~(\ref{eq3.1}), the second row follows from the substitution $u=e^t$ followed by~(\ref{eq3.2}), and the last row follows from the definitions of $R$ and $\phi$.
\end{proof}

Thus the efficient score is free of $h$, so to estimate $\theta$ efficiently, one does not need to estimate the covariate distribution. Hence, one does not need to impose an additional identifiability condition for $h$ such as the mean-zero assumption. The efficient information is 
\begin{equation}
\tilde{I}_{\tzero} = E\int_{0}^{\Uctzero}D(\tilZ,C,\theta_{0},s)D(\tilZ,C,\theta_{0},s)'(R\phi)^{2}(s)dF_{\Utzero}(s),
\label{eq2.8}
\end{equation}
where $D(\tilZ,C,\theta_{0},s)=[ \tilZ - E\{\tilZ | \ \Uctzero \ge s\} ]$. This is somewhat complicated to estimate, but the approach described in Remark~2 of \cite{Zeng07} will yield a consistent estimator which can be 
used for inference on $n^{1/2}(\hat{\theta}_n-\theta_0)$.

Since the backward recurrence times are uncensored, we can assume that the censoring times are infinite with $M(t) = I\{\Ut\le t\}$ and $Ra(t) = a(t)$. Thus 
the efficient score for the backward recurrence time simplifies to \begin{equation}
\tilde{l}_{\tzero,\lambda} = (\tilZ - E\tilZ)[1 - \Utzero \lambda\{\Utzero\}]
\label{eq2.10}
\end{equation}
and the efficient information becomes 
\begin{equation}
E\{ \tilde{l}_{\tzero}\tilde{l}_{\tzero}'\}=E\{ (\tilZ - E\tilZ)(\tilZ - E\tilZ)' \} E[1-\Utzero \lambda\{\Utzero\}]^{2}.
\label{eq2.11}
\end{equation}

Before presenting Lemma~\ref{lem1}, we provide a few needed definitions.
Let ${\cal Q}$ be the model consisting of densities on $\mathcal{R}^{+}$ of the form $S_g(u)/\int_0^{\infty}S_g(v)dv$, where $g\in{\cal G}$; and let $\dot{\cal Q}_S$ and $\dot{\cal G}_g$ be the respective tangent sets for ${\cal Q}$ and ${\cal G}$ at $S$ and at $g$, where $g$ satisfies $S_g=S$. Lemma~\ref{lem1} establishes that $\dot{Q}_S=L_2^0(S)$, which is needed in the proof of Theorem~\ref{thm:1} to identify $\dot{\mathcal{P}}_S$, a key technical step. We will be using score operators which allow us to construct scores for a model of interest from scores for a simpler model \citep[see, e.g., Chapter 18 of][]{K08}. 

\begin{lem}
	If  $A_{S}$ is the score operator mapping tangents in $\dot{\mathcal{G}}_{g}$ to $\dot{\mathcal{Q}}_{S}$, then $A_{S}\dot{\mathcal{G}}_{g}$ is dense in the maximal tangent set $L_{2}^{0}(S)$ for ${\cal Q}_S$, i.e.,
	$\dot{\cal Q}_S=L_2^0(S)$.\label{lem1}
\end{lem}

\begin{proof}
 Let $g$ be the density on $\mathcal{R}^{+}$ corresponding to $S$. Consider the following parametric path through $g$:
	\begin{equation*}
	\eta\mapsto g_{\eta}=\frac{\psi(\eta a)g}{\int\psi(\eta a)g},
	\end{equation*}
	where $\psi : \mathcal{R} \mapsto \mathcal{R}^{+}$ is bounded, continuously differentiable with bounded derivative $\psi'$ satisfying $\psi(0) = \psi'(0) = 1$ and $a\in L_{2}^{0}(g)$. Note that $L_{2}^{0}(g)$ is the closure within $L_2(g)$ of the derivatives of curves $g_{\eta}$ with respect to $\eta$ and $L_{2}^{0}(S)$ is the closure within $L_2(S)$ of derivatives of curves $\log(p_{g_{\eta}})$ with respect to $\eta$, where 
	\begin{eqnarray*}
		p_{g_{\eta}}(u)=\frac{S_\eta(u)}{\int S_\eta(t)dt},
	\end{eqnarray*}
	and $S_\eta$ is the survival function corresponding to $g_\eta$. 
	Thus, $L_{2}^{0}(g)$ is the maximal non-parametric tangent set for $\mathcal{G}$ while $L_{2}^{0}(S)$ is the maximal tangent set for $\mathcal{Q}_S$. The corresponding parametric submodel for $p_{g}$ is 
	\begin{equation*}
	p_{g_{\eta}}(u)= \frac{\int_{u}^{\infty} \psi(\eta a)(v)g(v)dv}{\int_{0}^{\infty}\int_{w}^{\infty}\psi(\eta a)(v)g(v)dv dw} .
	\end{equation*}
	Thus, the tangent set $\dot{\mathcal{Q}}_{S}$ (which consists of scores with respect to the one parameter models $p_{g_{\eta}}$) is given by the operator
	\[A_{S}a(u)=\frac{\int_{u}^{\infty}a(v)g(v)dv}{S(u)} -\int_{0}^{\infty}\frac{\int_{w}^{\infty}a(v)g(v)dv}{\int_{0}^{\infty}S(v)dv} dw.\]
	Let ${\cal A}$ be space of the bounded functions on $\mathcal{R}^{+}$ and ${\cal B}$ be the subset of ${\cal A}$ of functions which attain zero at all time points large enough. It is easy to verify that ${\cal A}$ is dense in $L_2^0(g)$ and that ${\cal B}$ is dense in $L_2^0(S)$, and, moreover, that $A_S a\in L_2^0(S)$ for all $a \in {\cal A} \cap L_2^0(g)$ and that $A_S^{\ast} b\in L_2^0(g)$ for all $b \in {\cal B} \cap L_2^0(S)$, where $A_S^{\ast}$ is the adjoint of $A_S$ defined as the solution to
	\[\langle A_S a, b\rangle_{L_2^0(S)}=\langle a, A_S^{\ast} b\rangle_{L_2^0(g)},\]
	for all $a\in {\cal A} \cap L_2^0(g)$ and $b\in {\cal B} \cap L_2^0(S)$. This relation yields that $A_S^{\ast}b=\int_0^u b(v)dv$. 
	
	By definition of $A_S$, $\dot{\mathcal{Q}}_{S}$ is the closed linear span of $A_S {\cal A}$ in $L_2^0(S)$, and thus 
	$\dot{\mathcal{Q}}_{S}\subset L_2^0(S)$. To prove the lemma, we need to verify that $L_2^0(S)\subset\dot{\mathcal{Q}}_{S}$ also holds. Suppose there is a $b_0\in L_2^0(S)$ which is not in $\dot{\mathcal{Q}}_{S}$. Then there exists a sequence $\{b_n\}\in{\cal B} \cap L_2^0(S)$ such that
	$\|b_n-b_0\|_{L_2^0(S)}\rightarrow 0$ and $0=$
	\begin{eqnarray*}
   \lim_{n\rightarrow\infty}\sup_{a\in{\cal A}:\, \|a\|_{L_2^0(g)}=1}\langle A_S a, b_n\rangle_{L_2^0(S)} & = \lim_{n\rightarrow\infty}\sup_{a\in{\cal A}:\, \|a\|_{L_2^0(g)}=1}\langle a,A_S^{\ast} b_n\rangle_{L_2^0(g)}\\
		  =	\lim_{n\rightarrow\infty}\|A_S^{\ast}b_n\|_{L_2^0(g)}.&
	\end{eqnarray*}

	This now implies that $\int_0^{\infty}\left\{\int_0^u b_n(v)dv\right\}^2 g(u)du\rightarrow 0$.  We can now show that for any $0<c<\infty$ for which $S(c)>0$,
	\begin{eqnarray*}
		\int_0^c\left\{\int_0^u b_0(v)dv\right\}^2 g(u)du&\leq&2\int_0^c\left\{\int_0^u b_n(v)dv\right\}^2g(u)du\\
		&&+2\int_0^c\left[\int_0^u\left\{b_n(v)-b_0(v)\right\}dv\right]^2g(u)du\\
		&\rightarrow& 0,
	\end{eqnarray*}
	as $n\rightarrow\infty$ by previous arguments combined with some analysis. Since $c$ was an arbitrary choice for which $S(c)>0$, we obtain that $b_0=0$ $S$-almost surely, and the desired conclusion follows.
\end{proof}

\subsection{Inference for Length-Biased Data}
A similar result may be obtained for length-biased data by replacing in the proof of Theorem \ref{thm:1} $S(u)/\int{S(u)d(u)}$ with $\tilde{g}=ug(u)/\int{S(u)d(u)}$, where, as before, $g$ is the density generating $S$. This yields the following result: 
\begin{thm}\label{thm:2}
	Using the same notation as Theorem \ref{thm:1} and under the same conditions, the efficient score for $\theta$ at $\theta=\tzero$ for length-biased data is
	\begin{equation}
	\tilde{l}_{\theta,S} = \int_{0}^{\Uctzero}[z - E\{\tilZ \mid \Uctzero \ge s\}]R \phi(s)dM(s),
	\label{eq2.9}
	\end{equation}
	where for $a\in L_{2}^{0}(S)$,
	\begin{equation*}
	Ra(t) = a(t) - \frac{\int_{t}^{\infty}a(u)ug(u)du}{\int_{t}^{\infty}ug(u)du}.
	\end{equation*}
	with $\phi(u) = 1- u g(u)/S(u)$,
	\begin{equation}
	M(t) = I\{\Ut\le t\} - \int_{0}^{t} I\{\Ut > s\}\lambda_{\Ut}(s)ds,
	\label{eq2.12}
	\end{equation}
	and \begin{equation}
	\lambda_{\Ut}(u) = \frac{ug\{e^{(\theta-\theta_{0})'z}u\}}{\int_{u}^{\infty} ug\{\ettzeroz w\}dw}.
	\label{eq2.13}
	\end{equation}
\end{thm}

\begin{proof}
	The proof of Theorem \ref{thm:2} is very similar to the proof of Theorem \ref{thm:1}. It follows along the same lines with a few minor differences, which are outlined below:
	The likelihood for one observation $(\Ui \wedge \Uci,\delta_{i},\tilZ_{i})$ is given by
	\begin{eqnarray*}
		l(\theta)=\left\{g_{\Ut}(U_{i})\right\}^{\delta_{i}}\left\{\int_{\Uci}^{\infty}g_{\Ut}(u)du\right\}^{1-\delta_{i}}h_{Z,\theta}(\tilZ_{i}).
	\end{eqnarray*}
	Here, the actual form of $g_{\Ut}(U_{i})=\ettzeroZi\tilde{g}(\ettzeroZi U_i)$ is different from Theorem~\ref{thm:1}. Thus we need to replace ${\cal P}_S$ with ${\cal P}_{\tilde{g}}$, where the map
	$g\mapsto\tilde{g}$ is as implicitly defined above just before the statement of Theorem~\ref{thm:2}. Specifically, this is the new model for $g$ holding $\theta$ and $h$ fixed at their true values. The other models and submodels are the same as for Theorem \ref{thm:1} except that $S$ and $S_0$ are replaced by $\tilde{g}$ and $\tilde{g}_0$. Taking log of $l(\theta)$ and differentiating with respect to $\theta$ 
	we obtain 
	$\dot{l}_{\theta_0}=\tilZ_i\left[\delta_{i} \phi(\Ui) + (1-\delta_{i}) E\left\{\phi(U(\theta_0))\mid U(\theta_0) > \Uci\right\}\right] + (\tilZ_{i} - E \tilZ)$ as the ordinary score for $\theta$ at $\theta=\theta_0$.
	The quantity in brackets on the right hand side is a stochastic integral with respect to the counting-process martingale in (\ref{eq2.12}) and is thus also a martingale. Using this, we can obtain the ordinary score
	\begin{equation}
	\dot{l}_{\tzero}= \tilZ\int_{0}^{\Uctzero} R\phi(s)dM(s) + (\tilZ - E\tilZ).
	\label{eq4.3}
	\end{equation} 
	
	Let ${\cal G}$, $\dot{\cal G}_g$, and the model ${\cal Q}'$ be as defined in Section 2 in the main text. Let $\dot{Q}_{\tilde{g}}'$ be the tangent set for ${\cal Q}'$ at $\tilde{g}$, and let $\dot{\cal P}_{\tilde{g}}$ be the tangent set corresponding to the new model ${\cal P}_{\tilde{g}}$. By using Lemma~\ref{lem2} below, which is similar to Lemma~\ref{lem1} but adapted to length-biased data, we can conclude that the tangent space $\dot{\mathcal{Q}}_{\tilde{g}}'$ can be taken to be the maximal tangent space $L_{2}^{0}(\tilde{g})$, and thus we obtain that the score operator $\dot{l}_{\tilde{g}}$ for $\tilde{g}$ is
	\begin{equation}
	\dot{l}_{\tilde{g}}b = \int_{0}^{\Uctzero} Rb(s)dM(s).
	\label{eq4.4}
	\end{equation}
	In order to find $\Pi_{0}(\dot{l}_{\tzero}\rvert \dot{\mathcal{P}}_{\tilde{g}}) = \dot{l}_{\tilde{g}}b^{*}$, we find $b^{*} \in L^{0}_{2}(\tilde{g})$ such that $\dot{l}_{\tzero} - \dot{l}_{\tilde{g}}b^{*} \ \perp \ \dot{l}_{\tilde{g}}b$ for all $b\in L_{2}^{0}(\tilde{g})$. That is $E\left\{ \left(\dot{l}_{\tzero} - \dot{l}_{\tilde{g}}b^{*}\right)\dot{l}_{\tilde{g}}b\right\}=0$. Note that $\dot{l}_{\tzero} - \dot{l}_{\tilde{g}}b^{*} = \int _{-\infty}^{\Uctzero} (\tilZ R\phi - Rb^{*})dM(s) + (\tilZ - E\tilZ)$.
	
	After this, the proof in Theorem \ref{thm:1} can be followed verbatim to obtain the desired result.
\end{proof}

We now present Lemma~\ref{lem2} required for the proof of Theorem~\ref{thm:2}. This is a modification of Lemma~\ref{lem1} for the length-biased setting.
\begin{lem}
	Consider the semi-parametric model $\mathcal{P}=\{P_{g}: g\in \mathcal{G}\}$, where the distribution $P_{g}$ has density $p_{g}(u) = ug(u)/\int S_{g}$ and $\mathcal{G}$ is a collection of densities on $\mathcal{R}^{+}$. Let  $\dot{\mathcal{G}}_{g}$ and $\dot{\mathcal{P}}_{g}$ be the tangent sets for the models $\mathcal{G}$ and $\mathcal{P}$ respectively at $g$. If  $A_{g}$ is the score operator mapping tangents in $\dot{\mathcal{G}}_{g}$ to $\dot{\mathcal{P}}_{g}$ then, $A_{g}\dot{\mathcal{G}}_{g}$ is dense in the maximal tangent set $L_{2}^{0}(S)$ for $\mathcal{P}$.\label{lem2}
\end{lem}

\begin{proof}
	Consider the following parametric path through $g$:
	\begin{equation*}
	\eta\mapsto g_{\eta}=\frac{\psi(\eta a)g}{\int\psi(\eta a)g},
	\end{equation*}
	where $\psi : \mathcal{R} \mapsto \mathcal{R}^{+}$ is bounded, continuously differentiable with bounded derivative $\psi'$ satisfying $\psi(0) = 0$ and $\psi'(0) = g(0)$ and $a\in L_{2}^{0}(g)$.	
	Thus, using similar notations as in Lemma 1, $L_{2}^{0}(g)$ is the maximal non-parametric tangent set for $\mathcal{G}$ while $L_{2}^{0}(\tilde{g})$ is the maximal tangent set for $\mathcal{Q}'$. The corresponding parametric submodel for $\tilde{g}$ will then be
	\begin{equation*}
	\tilde{g}_{\eta}(u)= \frac{u\psi(\eta a)g(u)}{\int_{0}^{\infty}v\psi(\eta a)g(v)dv}.
	\end{equation*}
	After this, the proof of Lemma~\ref{lem1} can be followed essentially verbatim to obtain the desired result. 
\end{proof}

\section{Simulation Studies}\label{sec:sim}

As our method does not rely on the covariate distribution, it is a special case of the model used in the paper by \cite{Zeng07} (since we assume that the covariates are constant over time). So, we use their profile likelihood approach 
to estimate $\theta$ and compare it with Klaassen's mean zero approach and also the known covariate structure approach. 
We consider only 1 covariate Z $\sim\mbox{Unif}(-1,1)$. So, $\tilde{Z}$ has density given by $\theta e^{\theta z}/(e^{\theta}-e^{-\theta})$, where $-1\leq z \leq 1$.
We take different values of  $\theta$ and assume that the error distribution is standard normal, i.e., $U$ is lognormal.  Then, we use all three methods to estimate $\theta$. For the profile-likelihood approach, we use the Gaussian kernel and a bandwidth of $h_{n}=Qn^{-1/5}$, where Q is the interquartile range of the data. We consider 1000 replicates and look at the mean bias and variance in estimating $\theta$. We also look at what happens when the covariate distribution is misspecified. For this, we consider Z$\sim\mbox{Unif}(x,1)$ for some choice of x. So,  $\tilde{Z}$ has density given by $\theta e^{\theta z}/(e^{\theta}-e^{-\theta x})$ where x$\leq$z$\leq$1. We take the values $x=-0.9$ and $-0.8$  and compare both the mean zero and  known covariate distribution, assuming Z $\sim\mbox{Unif}(-1,1)$. We take $\theta$=1 for these simulations. We consider 1000 replicates in this case as well. The results are given in Table \ref{tab:one}.

\begin{table}[ht]
	\begin{center}
		\caption{Estimates for the Backward Recurrence Time Data}
		\label{tab:one}
		\begin{tabular}{ccrrcrrrc}
			\hline
			Parameters & Sample &\multicolumn{3}{c}{ Profile Likelihood Approach} & \multicolumn{2}{c}{Vanishing Mean} & \multicolumn{2}{c}{Known Covariance}\\
			& Size & \multicolumn{1}{c}{Bias}  &  \multicolumn{1}{c}{SE} & \multicolumn{1}{c}{CP (\%)} & \multicolumn{1}{c}{Bias} & \multicolumn{1}{c}{SE} & \multicolumn{1}{c}{Bias} & \multicolumn{1}{c}{SE}\\ 
			\hline
			
			$\theta=1$ & 100 & $-0.033$ & 0.258 & 94.4& $-0.110$ & 0.230 & $-0.051$ & 0.201\\
			& 200 & 0.006& 0.188& 95.6& $-0.052$ & 0.161 & $-0.001$ & 0.137\\
			& 400 & 0.010& 0.133 &95.1 & 0.002 & 0.111 & 0.001 & 0.099\\
			\hline
			$\theta=0.5$ &100 & 0.028 & 0.202 & 95.8 &  $-0.165$ & 0.190 & $-0.020$ & 0.142\\
			& 200 & 0.016 & 0.188& 95.1 & $-0.033$ & 0.134 & 0.007 & 0.095\\
			& 400 & 0.010 & 0.117& 95.5  & 0.015 & 0.092 & 0.004 & 0.067\\
			\hline
			$\theta=2$ & 100 & 0.010 & 0.299 & 93.9 & 0.027 & 0.553 & 0.027 & 0.366\\ 
			& 200  & 0.003& 0.203& 95.7 & 0.055 & 0.284 & 0.004 & 0.216\\ 
			& 400 & $-0.005$ & 0.169 & 94.5 & 0.012 & 0.198 & $-0.003$ & 0.105\\
			\hline
			$x=-0.9$ & 100 & $-0.029$ & 0.254 & 95.2 &  $-0.055$ & 0.250 & 0.031 & 0.222 \\
			$\theta=1$& 200 &0.009 & 0.181 & 94.1 & 0.181 & 0.171 & 0.115 & 0.138\\
			& 400 & 0.002 & 0.124 & 94.6 & 0.155 & 0.114 & 0.213 & 0.094\\
			\hline
			$x=-0.8$ &100 & 0.023 & 0.269 & 93.3 &  0.335 & 0.270 & 0.195 & 0.211\\
			$\theta=1$& 200 & $-0.006$ & 0.169 & 94.9 & 0.490 & 0.184 & 0.234 & 0.142\\
			& 400 & 0.003 & 0.108 & 94.4 & 0.304 & 0.127 & 0.278 & 0.099\\
			
			\hline\hline
		\end{tabular}
	\end{center}
\end{table}

Thus we find that the estimates obtained using our methods are quite comparable to the special case where the covariance structure is known, although Klaassen's method has lower variance. This is expected because Klaassen's method is under some additional model assumptions which are not used by our method.  
However, their estimates are very sensitive to model specification. On the other hand, our naive analysis yields unbiased estimates in both cases.  The variance estimators accurately reflect the actual variance, while the confidence intervals also have correct coverage probabilities.

\section{Data Analysis}\label{sec:dat}

For illustration, we analyze a subset of the data used by \citep{Keiding12}. It is a backward recurrence time data setting on the time to pregnancy obtained from a large French telephone survey. Women were eligible if they were between 18-44 years old, were living with a male partner and did not use any method to avoid pregnancy. We consider only nulliparous women who had not initiated  any fertility treatment. The response variable was the current duration of unprotected intercourse, which is the time elapsed from the start of unprotected intercourse and the interview.  
The estimates obtained for the covariates along with the 95 $\%$ confidence intervals are given in Table 2. We note that the naive estimator can accurately determine the effect of the covariates and is comparable with the ordinary least squares results. 

\begin{table}[ht]
	\begin{center}
		\caption{Estimates for time ratios and the corresponding confidence intervals for nulliparous women}
		\label{tab:two}
		\begin{tabular}{l ccc}
			\hline
			
			&& Semiparametric AFT & OLS\\
			
			Covariate & No & Time Ratio & Time Ratio\\
			\hline
			\multicolumn{4}{l}{Tobacco Consumption at recruitment}\\
			Non-Smokers & 159 & 1 & 1\\
			Smokers& 92 & 1.20(0.75,1.78) &  1.04(0.70,1.53)\\
			\multicolumn{4}{l}{Age at recruitment}\\
			0-17 &3 & 7.50(1.50,38.0) & 7.32(1.29,41.4)\\
			18-24 & 50 & 2.00(1.20,3.41) & 2.08(1.24,3.49)\\
			25-29 & 93 & 1 & 1\\
			30-34 & 62 & 1.00(0.61,1.74) & 1.01(0.63,1.64)\\
			35-39 & 41 & 1.10(0.61,2.02) & 0.93(0.54,1.62)\\
			40-44 & 2 & 0.13(0.01,1.17) & 0.13(0.02,1.10)\\
			\multicolumn{4}{l}{Frequency of Sexual Intercourse}\\
			$<$1 per month & 0\\
			1-3 per month & 44 & 2.20(1.20,3.89) &  2.18(1.27,3.71)\\
			1-2 per week & 109 & 1.20(0.78,1.92) & 1.23(0.81,1.86)\\
			$\geq$3 per week & 98 & 1 & 1\\
			\multicolumn{4}{l}{Menstrual Cycle Length}\\
			$<$27 days & 53 & 1& 1\\
			27-29 days & 110 &  0.90(0.52,1.55) & 0.86(0.52,1.41)\\
			$\geq$30 days & 88 & 1.10(0.62,1.81) & 0.99(0.59,1.63)\\
			\hline\hline\\
		\end{tabular}
	\end{center}
\end{table}

\section{Discussion}\label{sec:disc}

The assumption that the initiating event follows a homogeneous Poisson process has been widely used \citep{Cox69,VanEs00,Keiding02}. Such an assumption may be reasonable over short time periods, where the rate of the initiating event may be stable. Over longer time periods, where the assumption may not be valid, alternative methods may be needed. This is a challenging problem worthy of further investigation.

In the presence of covariates, a popular alternative to the accelerated failure time model is the proportional hazards model \citep{Cox72} given by $\lambda_{T \mid Z}(t) = e^{\tZ}\lambda(t)$,
where $\lambda_{T \mid Z}$ is the hazard function of $T$ given the covariate vector $Z$ and $\lambda$ is an unspecified baseline hazard function. Here the density of $T$ is given by $e^{\tz}\lambda(t)e^{-e^{\tz}\Lambda(t)}$, where, $\Lambda$ is the cumulative baseline hazard function satisfying $\Lambda(0) = 0$.
If we assume a proportional hazards model for $T$, then by (\ref{eq1.1}), under length-biased and cross-sectional sampling, the conditional density of the forward or the backward recurrence time $\Tbf$, given $\tilde{Z}$, is
\begin{equation*}
g_{\Tbf\mid \tilde{Z}=z}(t) = \frac{e^{-e^{\tz}\Lambda(t)}}{\int e^{-e^{\tz}\Lambda(t)}}, \quad t \in (0,\infty).
\end{equation*}
and the conditional hazard function is
\begin{equation*}
\lambda_{\Tbf \mid \tilde{Z}=z}(t) = \frac{e^{-e^{\tz}\Lambda(t)}}{\int_{t}^{\infty}e^{-e^{\tz}\Lambda(u)}du}.
\end{equation*}
Thus, going from $T$ to $\Tbf$, the proportional hazard structure is lost, unless either the baseline hazard is constant or $T$ given $Z$ follows a Pareto distribution \citep{VanEs00}. The usual techniques for fitting the proportional hazards model may yield biased results. Furthermore, it is unclear whether information in the observed covariates may be employed to yield more efficient estimation, in contrast to our results for the accelerated lifetime model for T given Z. Analogous issues arise for the length-biased time $T_{\mathrm{LB}}$ under the proportional hazards model. Further investigation is needed in these areas. 

Under additional assumptions on the covariate distribution, for example, $Z$ is mean zero or $h$ is known or specified parametrically, information may be gained using the observed covariates \citep{VanEs00}. The trade-off for these efficiency gains is the potential bias associated with the additional modeling assumptions for the distribution of $Z$. In our simulations given above, we found that the efficient estimators are rather sensitive to the extra assumptions and that the gains in efficiency when such assumptions are correctly specified are rather modest compared to the naive efficient estimators. 

\bibliographystyle{imsart-nameyear}
\bibliography{bib2}

\end{document}